\documentclass[a4paper, twoside,12pt]{article}
\usepackage{fancyhdr}
\usepackage{fnpos}
 \usepackage[english]{babel}        
\usepackage[T1]{fontenc}          
\usepackage{graphicx}             
\usepackage{makeidx}
\usepackage{fancybox}
\usepackage{framed}
\usepackage{fancyhdr}
 \usepackage{pstricks,pst-plot,pstricks-add}
\usepackage[margin=1in]{geometry}
\usepackage{graphicx}
\usepackage{titlesec}
\usepackage{amsmath}
\usepackage{amsfonts}   
\usepackage{amssymb}    
\usepackage{amsthm}
\usepackage{dsfont}
\usepackage{mathtools}
\usepackage{verbatim}   
\usepackage{color}
\usepackage{listings}
\usepackage{graphicx}
\usepackage{amsmath}
\usepackage{amsmath,amssymb,color,inputenc,euscript,graphicx,psfrag}

\usepackage{amsmath, amsthm, amscd, amsfonts, amssymb, graphicx, color,   mathrsfs}
\usepackage[english]{babel}
    \usepackage[bookmarksnumbered, colorlinks, plainpages]{hyperref}
    \hypersetup{colorlinks=true,linkcolor=blue, anchorcolor=blue, citecolor=blue, urlcolor=red, filecolor=magenta, pdftoolbar=true}

\newtheorem{thm}{Theorem}[section]

\newtheorem{lem}[thm]{Lemma}
\newtheorem{prop}[thm]{Proposition}
\newtheorem{defn}[thm]{Definition}
\newtheorem{rem}[thm]{Remark}
\numberwithin{equation}{section}

\renewcommand{\thefootnote}

\renewcommand\Im{\operatorname{Im}}

 {\markboth{\sl A. Saoudi }{\sl A variation of the$L^p$ uncertainty principles }

\author {Ahmed Saoudi}

\title{ A variation of the $L^p$ uncertainty principles for the Weinstein transform}

\date{}
 
\begin{document}
 \maketitle
\begin{center}
     Northern Border University, College of Science, Arar, P.O. Box 1631, Saudi Arabia.\\
   Universit\'{e} de Tunis El Manar, Facult\'{e} des sciences de Tunis,\\ Laboratoire de Fonctions Sp\'{e}ciales,
        analyse Harmonique et Analogues,\\ LR13ES06, 2092 El Manar I, Tunisie.\\
     \textbf{ e-mail:} ahmed.saoudi@ipeim.rnu.tn
\end{center}
  \begin{abstract}
The Weinstein operator  has several applications in pure and applied Mathematics especially in Fluid Mechanics and  satisfies some uncertainty principles similar to the Euclidean Fourier transform. The aim of this paper is establish a generalization of uncertainty principles for Weinstein transform in $L_\alpha^p$-norm. Firstly, we extend the Heisenberg-Pauli-Weyl uncertainty principle to more general case. Then we establish three continuous uncertainty principles of concentration type. The first and the second uncertainty principles are $L_\alpha^p$ versions and depend on the sets of concentration $\Omega$ and $\Sigma$, and on the time function $\varphi$. However, the third uncertainty principle is also $L_\alpha^p$ version depends on the sets of concentration and he is independent on the band limited function $\varphi$. These $L_\alpha^p$-Donoho-Stark-type inequalities generalize the results obtained in the case $p=q=2$.\\ \\

 \end{abstract}
  \section{ Introduction}
Uncertainty principles appear in harmonic analysis and signal theory in a variety of different forms
involving not only the  signal $\varphi$  and its Fourier transform $\mathcal{F}(\varphi)$, but essentially every representation of a signal in the time-frequency space. They  are mathematical results that give limitations on the simultaneous concentration of a signal and its Fourier transform and they have implications in signal analysis and  quantum physics. In quantum physics they tell us that a particle's speed and position cannot both be measured with infinite precision. In signal analysis they tell us that if we observe a signal only for a finite period of time, we will lose information about the frequencies the signal consists of. Timelimited functions and bandlimited functions are basic tools of signal and image
processing. Like, the simple form of the uncertainty principle tells us that a signal cannot be simultaneously time and bandlimited. This leads to the investigation
of the set of almost time and almost bandlimited functions, which has been initially
by Landau, Pollak \cite{landau1975szego, landau1962prolate} and then by Donoho, Stark \cite{donoho1989uncertainty}.

Motivated by the work of Laeng and Morpurgo \cite{laeng1999uncertainty, morpurgo2001extremals} and the work of Soltani and Ghazwani \cite{zbMATH06504466, zbMATH06692262}, we propose an extension of the techniques and results of Ghobber \cite{ghobber2013uncertainty, ghobber2014variations} to establish a variation of the $L^p$ uncertainty principles for the Weinstein transform.

Many uncertainty principles have already been proved for the Weinstein transform $\mathcal{F}_{W,\alpha}
$,  \cite{mejjaoli2012weinstein, mejjaoli2011uncertainty, salem2015heisenberg}. The authors have established in \cite{mejjaoli2011uncertainty} the Heisenberg-Pauli-Weyl inequality for the Weinstein transform, by showing that, for every $\varphi$ in $L^2_\alpha(\mathbb{R}^{d+1}_+)$
\begin{equation}\label{firstuncert}
\|\varphi\|_{\alpha, 2}\leq \frac{2}{2\alpha+d+2}
\||x|\varphi\|_{\alpha, 2}\||y|\mathcal{F}_{W,\alpha}(\varphi)\|_{\alpha, 2}.
\end{equation}

Our first result will be the following variation of Heisenberg-Pauli-Weyl-type inequality.

 \medskip
 \noindent{\bf Theorem A.}\
 {\sl  Let $1<p\leq 2$, $q=p/(p-1)$,  $0<s<(2\alpha+d+2)/q$ and $t>0$, then for all $\varphi\in L^p_\alpha(\mathbb{R}^{d+1}_+)$,
 \begin{equation*}\label{LpHPWU}
   \left\|\mathcal{F}_{W,\alpha}(\varphi)\right\|_{\alpha,q}  \leq K(s,t) \left\||x|^s\varphi\right\|_{\alpha,p}^\frac{t}{s+t}\left\||y|^t\mathcal{F}_{W,\alpha}(\varphi)\right\|_
   {\alpha,q}^\frac{s}{s+t},
 \end{equation*}
 where $K(s,t)$ is a positive constant.}

 This theorem implies in particular that if $\varphi$ is highly localized in the neighbourhood of $x=0$,
then $\mathcal{F}_{W,\alpha}(\varphi)$ cannot be concentrated in the neighbourhood of $y=0$. In particular, for $p=q=2$, we obtain the general case of Heisenberg-Pauli-Weyl-type inequality (\ref{firstuncert})
 \begin{equation*}
   \left\|\mathcal{F}_{W,\alpha}(\varphi)\right\|_{\alpha,2}  \leq K(s,t) \left\||x|^s\varphi\right\|_{\alpha,2}^\frac{t}{s+t}\left\||y|^t\mathcal{F}_{W,\alpha}(\varphi)\right\|_
   {\alpha,2}^\frac{s}{s+t}.
 \end{equation*}

The second and third results are two continuous-time uncertainty principles of concentration type depend on the sets of concentration $\Omega$ and $\Sigma$, and on the time function $\varphi$.

 \medskip
 \noindent{\bf Theorem B.}\
  {\sl  Let $\Omega$, $\Sigma$ be a measurable subsets of $\mathbb{R}^{d+1}_+$ and  $\varphi\in L^p_\alpha(\mathbb{R}^{d+1}_+)$, $1\leq p\leq 2$. If $\varphi$ is $\varepsilon_\Omega$-concentrated to $\Omega$ in $ L^p_\alpha(\mathbb{R}^{d+1}_+)$-norm and  $\mathcal{F}_{W,\alpha}(\varphi)$ is $\varepsilon_\Sigma$-concentrated to $\Sigma$ in $ L^q_\alpha(\mathbb{R}^{d+1}_+)$-norm, $q=p/(p-1)$, then we have
 \begin{equation*}
    \left\|\mathcal{F}_{W,\alpha}(\varphi)\right\|_{\alpha,q}\leq \frac{ \left(\mu_{\alpha}(\Sigma)\right)^\frac{1}{q} \left(\mu_{\alpha}(\Omega)\right)^\frac{1}{q}+\varepsilon_\Omega}{1-\varepsilon_\Sigma}\left\|\varphi
    \right\|_{\alpha,p}.
 \end{equation*} }

   The statement of Theorem B depends on the time function $\varphi$. However although for $p=q=2$, the continuous-time uncertainty principle becomes
 \begin{equation*}
    1-\varepsilon_\Omega-\varepsilon_\Sigma\leq \mu_\alpha(\Omega)^\frac{1}{2}\mu_\alpha(\Sigma)^\frac{1}{2},
  \end{equation*}

 \medskip
 \noindent{\bf Theorem C.}\
 {\sl   Let $\Omega$, $\Sigma$ be a measurable subsets of $\mathbb{R}^{d+1}_+$ and  $\varphi\in (L^1_\alpha\cap L^p_\alpha)(\mathbb{R}^{d+1}_+)$, $1\leq p\leq 2$. If $\varphi$ is $\varepsilon_\Omega$-concentrated to $\Omega$ in $ L^p_\alpha(\mathbb{R}^{d+1}_+)$-norm and  $\mathcal{F}_{W,\alpha}(\varphi)$ is $\varepsilon_\Sigma$-concentrated to $\Sigma$ in $ L^q_\alpha(\mathbb{R}^{d+1}_+)$-norm, $q=p/(p-1)$, then we have
   \begin{equation*}
    \left\|\mathcal{F}_{W,\alpha}(\varphi)\right\|_{\alpha,q}\leq \frac{ \left(\mu_{\alpha}(\Sigma)\right)^\frac{1}{q} \left(\mu_{\alpha}(\Omega)\right)^\frac{1}{q}}{(1-\varepsilon_\Omega)(1-\varepsilon_\Sigma)}\left\|\varphi
    \right\|_{\alpha,p}.
 \end{equation*} }

 Likewise, the statement of the pervious Theorem depends on the time function $\varphi$ and Like the first  continuous-time uncertainty principle(Theorem B) is independent on $\varphi$ for $p=q=2$, and we have
 \begin{equation*}
    (1-\varepsilon_\Omega)(1-\varepsilon_\Sigma)\leq \mu_\alpha(\Omega)^\frac{1}{2}\mu_\alpha(\Sigma)^\frac{1}{2}.
  \end{equation*}

The last result is a continuous-bandlimited uncertainty principle of
concentration type depends on the sets of concentration
$\Omega$ and $\Sigma$, but he is independent on the bandlimited function.


 \medskip
 \noindent{\bf Theorem D.}\
  Let $\Omega$, $\Sigma$ be a measurable subsets of $\mathbb{R}^{d+1}_+$ and let $\varphi\in L^p(\mathbb{R}_+^d)$ such that  $1\leq p\leq 2$. Then if $\varphi$ is  $\varepsilon_\Omega$-concentrated to $\Omega$ and $\varepsilon_\Sigma$-bandlimited to $\Sigma$ in $L^p_\alpha$-norm, we have
  \begin{equation*}
     1-\varepsilon_\Omega-\varepsilon_\Sigma\leq  (1+\varepsilon_\Sigma)\left(\mu_{\alpha}(\Sigma)\right)^\frac{1}{p} \left(\mu_{\alpha}(\Omega)\right)^\frac{1}{p}.
   \end{equation*}

The main body of the paper is organized as follows. In section 2, we recall some harmonic analysis results related to the Weinstein operator. In section 3, we prove a variation of  Heisenberg-Pauli-Weyl uncertainty principle for the Weinstein operator. Finally in section 4, we establish three continuous uncertainty principles of concentration type  in  $L^p_\alpha$-norm. These estimations depend on the sets of concentration $\Omega$ and $\Sigma$, and on  the band-limited function $\varphi$, only the last estimation is
independent on $\varphi$.

 \section{Preliminaires}
  The Weinstein operator $\Delta_{W,\alpha}^d$ defined on $\mathbb{R}_{+}^{d+1}=\mathbb{R}^d\times(0, \infty)$, by
\begin{equation*}
\Delta_{W,\alpha}^d=\sum_{j=1}^{d+1}\frac{\partial^2}{\partial x_j^2}+\frac{2\alpha+1}{x_{d+1}}\frac{\partial}{\partial x_{d+1}}=\Delta_d+L_\alpha,\;\alpha>-1/2,
\end{equation*}
where $\Delta_d$ is the Laplacian operator for the $d$ first variables and $L_\alpha$ is the Bessel operator for the last variable defined on $(0,\infty)$ by
$$L_\alpha u=\frac{\partial^2 u}{\partial x_{d+1}^2}+\frac{2\alpha+1}{x_{d+1}}\frac{\partial u}{\partial x_{d+1}}.$$

The Weinstein operator $\Delta_{W,\alpha}^d$ has several applications in pure and applied mathematics, especially in fluid mechanics \cite{brelot1978equation, weinstein1962singular}.

For all $\lambda=(\lambda_1,...,\lambda_{d+1})\in\mathbb{C}^{d+1}$, the system
\begin{equation}
\begin{gathered}
\frac{\partial^2u}{\partial x_{j}^2}(  x)
  =-\lambda_{j} ^2u(x), \quad\text{if } 1\leq j\leq d \\
L_{\alpha}u(  x)  =-\lambda_{d+1}^2u(  x), \\
u(  0)  =1, \quad \frac{\partial u}{\partial
x_{d+1}}(0)=0,\quad \frac{\partial u}{\partial
x_{j}}(0)=-i\lambda_{j}, \quad \text{if } 1\leq j\leq d
\end{gathered}
\end{equation}
 has a unique solution  denoted by $\Lambda_{\alpha}^d(\lambda,.),$ and given by
\begin{equation}\label{wkernel}
\Lambda_{\alpha}^d(\lambda,x)=e^{-i<x^\prime,\lambda^\prime>}j_\alpha(x_{d+1}\lambda_{d+1})
\end{equation}
 where $x=(x^\prime,x_{d+1}),\; \lambda=(\lambda^\prime,\lambda_{d+1})$ and $j_\alpha$ is is the normalized Bessel function of index $\alpha$ defined by
$$j_\alpha(x)=\Gamma(\alpha+1)\sum_{k=0}^\infty\frac{(-1)^k x^{2k}}{2^k k!\Gamma(\alpha+k+1)}.$$
The function $(\lambda,x)\mapsto\Lambda_{\alpha}^d(\lambda,x)$ is called the Weinstein kernel and has a unique extension to $\mathbb{C}^{d+1}\times\mathbb{C}^{d+1}$, and satisfied the following properties.\\
\begin{itemize}
\item[(i)] For all $(\lambda,x)\in \mathbb{C}^{d+1}\times\mathbb{C}^{d+1}$ we have
\begin{equation*}
\Lambda_{\alpha}^d(\lambda,x)=\Lambda_{\alpha}^d(x,\lambda).
\end{equation*}
\item[(ii)] For all $(\lambda,x)\in \mathbb{C}^{d+1}\times\mathbb{C}^{d+1}$ we have
\begin{equation*}
\Lambda_{\alpha}^d(\lambda,-x)=\Lambda_{\alpha}^d(-\lambda,x).
\end{equation*}
\item[(iii)] For all $(\lambda,x)\in \mathbb{C}^{d+1}\times\mathbb{C}^{d+1}$ we get
\begin{equation*}
\Lambda_{\alpha}^d(\lambda,0)=1.
\end{equation*}
\item[(iv)] For all $\nu\in\mathbb{N}^{d+1},\;x\in\mathbb{R}^{d+1}$ and $\lambda\in\mathbb{C}^{d+1}$ we have
\begin{equation*}\label{klk}
 \left|D_\lambda^\nu\Lambda_{\alpha}^d(\lambda,x)\right|\leq\left\|x\right\|^{\left|\nu\right|}e^{\left\|x\right\|\left\|\Im \lambda\right\|}
\end{equation*}
\end{itemize}
where $D_\lambda^\nu=\partial^\nu/(\partial\lambda_1^{\nu_1}...\partial\lambda_{d+1}^{\nu_{d+1}})$ and $\left|\nu\right|=\nu_1+...+\nu_{d+1}.$ In particular, for all $(\lambda,x)\in \mathbb{R}^{d+1}\times\mathbb{R}^{d+1}$, we have
\begin{equation}\label{normLambda}
\left|\Lambda_{\alpha}^d(\lambda,x)\right|\leq 1.
\end{equation}

In the following we denote by
\begin{itemize}
\item[(i)] $C_*(\mathbb{R}^{d+1})$, the space of continuous functions on $\mathbb{R}^{d+1},$ even with respect to the last variable.
\item[(ii)] $S_*(\mathbb{R}^{d+1})$, the space of the $C^\infty$ functions, even with respect to the last variable, and rapidly decreasing together with their derivatives.
	\item[(iii)] $L^p_\alpha(\mathbb{R}^{d+1}_+),\;1\leq p\leq \infty,$ the space of measurable functions $f$ on $\mathbb{R}^{d+1}_+$ such that
	$$\left\|f\right\|_{\alpha,p}=\left(\int_{\mathbb{R}^{d+1}_+}\left|f(x)\right|^pd\mu_\alpha(x)\right)^{1/p}<\infty, \;p\in[1,\infty),$$
	$$\left\|f\right\|_{\alpha,\infty}=\textrm{ess}\sup_{x\in\mathbb{R}^{d+1}_+}\left|f(x)\right|<\infty,$$

where $d\mu_{\alpha}(x)$ is the measure on  $\mathbb{R}_{+}^{d+1}=\mathbb{R}^d\times(0,\infty)$ given by
\begin{equation*}\label{mesure}
	d\mu_\alpha(x)=\frac{x^{2\alpha+1}_{d+1}}{(2\pi)^d2^{2\alpha}\Gamma^2(\alpha+1)}dx.
	\end{equation*}
\end{itemize}

For a radial function $\varphi\in L_{\alpha}^{1}(\mathbb{R}_{+} ^{d+1})$ the function $\tilde{\varphi}$ defined on $\mathbb{R}_+$ such that $\varphi(x)=\tilde{\varphi}(|x|)$, for all
$x\in\mathbb{R}_{+} ^{d+1}$, is integrable with respect to the measure $r^{2\alpha+d+1}dr$, and we have
\begin{equation}\label{radialweinstein}
  \int_{\mathbb{R}_{+}^{d+1}}\varphi(x)d\mu_{\alpha}(x)=a_\alpha\int_{0}^{\infty}
  \tilde{\varphi}(r)r^{2\alpha+d+1}dr,
\end{equation}
where $$a_\alpha=\frac{1}{2^{\alpha+\frac{d}{2}}\Gamma(\alpha+\frac{d}{2}+1)}.$$
The Weinstein transform generalizing the usual Fourier transform, is given for
$\varphi\in L_{\alpha}^{1}(\mathbb{R}_{+} ^{d+1})$ and $\lambda\in\mathbb{R}_{+}^{d+1}$, by
$$
\mathcal{F}_{W,\alpha}
(\varphi)(\lambda)=\int_{\mathbb{R}_{+}^{d+1}}\varphi(x)\Lambda_{\alpha}^d(x, \lambda
)d\mu_{\alpha}(x),
$$

We list some known basic properties of the Weinstein transform are as follows. For the proofs, we refer \cite{nahia1996mean, nahia1996spherical}.

\begin{itemize}
	\item[(i)] For all $\varphi\in L^1_\alpha(\mathbb{R}^{d+1}_+)$, the function $\mathcal{F}_{W,\alpha}(\varphi)$  is continuous on $\mathbb{R}^{d+1}_+$ and we have
	\begin{equation}\label{L1-Linfty}
	\left\|\mathcal{F}_{W,\alpha}\varphi\right\|_{\alpha,\infty}\leq\left\|\varphi\right\|_{\alpha,1}.
	\end{equation}
	\item[(ii)]   The Weinstein transform is a topological isomorphism from $\mathcal{S}_*(\mathbb{R}^{d+1}_+)$ onto itself. The inverse transform is given by
	\begin{equation}\label{inversionweinstein}
	\mathcal{F}_{W,\alpha}^{-1}\varphi(\lambda)= \mathcal{F}_{W,\alpha}\varphi(-\lambda),\;\textrm{for\;all}\;\lambda\in\mathbb{R}^{d+1}_+.
	\end{equation}
	
	\item[(iii)] Parseval formula: For all $\varphi, \phi\in \mathcal{S}_*(\mathbb{R}^{d+1}_+)$, we have
	\begin{equation*}\label{MM} \int_{\mathbb{R}^{d+1}_+}\varphi(x)\overline{\phi(x)}d\mu_\alpha(x)=\int_{\mathbb{R}^{d+1}_+}\mathcal{F}_{W,\alpha}
(\varphi)(x)\overline{\mathcal{F}_{W,\alpha}(\phi)(x)}d\mu_\alpha(x).
	\end{equation*}
\item[(v)] Plancherel formula: For all $\varphi\in L^2_\alpha(\mathbb{R}^{d+1}_+)$, we have
\begin{equation}\label{Plancherel formula}
	\left\|\mathcal{F}_{W,\alpha}\varphi\right\|_{\alpha,2}=\left\|\varphi\right\|_{\alpha,2}.
	\end{equation}
\item[(vi)] Plancherel Theorem: The Weinstein transform $\mathcal{F}_{W,\alpha}$ extends uniquely to an isometric isomorphism on $L^2_\alpha(\mathbb{R}^{d+1}_+).$
\item[(vii)] Inversion formula: Let $\varphi\in L^1_\alpha(\mathbb{R}^{d+1}_+)$ such that $\mathcal{F}_{W,\alpha}\varphi\in L^1_\alpha(\mathbb{R}^{d+1}_+)$,  then we have
\begin{equation}\label{inv}
\varphi(\lambda)=\int_{\mathbb{R}^{d+1}_+}\mathcal{F}_{W,\alpha}\varphi(x)\Lambda_{\alpha}^d(-\lambda,x)d\mu_\alpha(x),\;\textrm{a.e. }\lambda\in\mathbb{R}^{d+1}_+.
\end{equation}
\end{itemize}

Using relations (\ref{L1-Linfty}) and (\ref{Plancherel formula}) with Marcinkiewicz's interpolation theorem \cite{zbMATH03367521} we deduce that for every $\varphi\in L^p_\alpha(\mathbb{R}^{d+1}_+)$ for all $1\leq p\leq 2$, the function $\mathcal{F}_{W,\alpha}(\varphi)\in L^q_\alpha(\mathbb{R}^{d+1}_+), q=p/(p-1),$ and
\begin{equation}\label{Lp-Lq}
	\left\|\mathcal{F}_{W,\alpha}\varphi\right\|_{\alpha,q}\leq\left\|\varphi\right\|_{\alpha,p}.
	\end{equation}

\section{$L^p$-Heisenberg-Pauli-Weyl inequality}
In this section, we extend the Heisenberg-Pauli-Weyl uncertainty principle (\ref{firstuncert})
to more general case for the Weinstein operator. We need to use the method of Ciatti et al. \cite{ciatti2007heisenberg}, which is the
counterpart in the Euclidean case. In the following we need this lemma.
\begin{lem}
  Let $1<p\leq 2$, $q=p/(p-1)$ and $0<s<(2\alpha+d+2)/q$. Then for all $\varphi\in L^p_\alpha(\mathbb{R}^{d+1}_+)$ and $z>0$,

  \begin{equation}\label{lemexpo}
    \left\|e^{-z|y|^2}\mathcal{F}_{W,\alpha}(\varphi)\right\|_{\alpha,q}\leq \left(1+\frac{K_\alpha}{(2q)^{(\alpha+\frac{1}{d}+1)}\frac{1}{q}}\right)z^{-s/2} \left\||x|^s\varphi\right\|_{\alpha, p},
  \end{equation}
  where
  $$K_\alpha=\left((2\alpha+d+2-qs)2^{\alpha+\frac{d}{2}}\Gamma(\alpha+\frac{d}{2}+1)\right)^{-1/q}$$
\end{lem}
\begin{proof}
   Let  $\varphi \in  L^p_{\alpha}(\mathbb{R}^{d+1}_+)$. The inequality (\ref{lemexpo}) holds if
   $\left\||x|^s\varphi\right\|_{\alpha, p}=\infty.$
    Let us now assume that $\left\||x|^s\varphi\right\|_{\alpha, p}<\infty.$ For $\rho>0$, we put $B_\rho=\{x\in\mathbb{R}^{d+1}_+: |x|<\rho\}$ and $B_\rho^c=\mathbb{R}^{d+1}_+\backslash B_\rho.$ Denote by $\chi_{B_\rho}$ and $\chi_{B_\rho^c}$ the characteristic functions.  Let $\varphi\in L^p_\alpha(\mathbb{R}^{d+1}_+)$, $1<p\leq 2$ and $q=p/(p-1)$. Like that
    $$|\varphi\chi_{B_\rho^c}(x)|\leq \rho^{-s}|x|^s|\varphi(x)|,$$
    then, by inequality (\ref{Lp-Lq}), we get
    \begin{align*}
       \left\|e^{-z|y|^2}\mathcal{F}_{W,\alpha}(\varphi\chi_{B_\rho^c})\right\|_{\alpha,q}\leq  & \left\|e^{-z|y|^2}\right\|_{\alpha,\infty}\left\|\mathcal{F}_{W,\alpha}(\varphi\chi_{B_\rho^c})
       \right\|_{\alpha,q} \\
      \leq & \left\|e^{-z|y|^2}\right\|_{\alpha,\infty}\left\|\varphi\chi_{B_\rho^c}
       \right\|_{\alpha,p} \\
     \leq  & \left\|e^{-z|y|^2}\right\|_{\alpha,\infty}\rho^{-s}\left\||x|^s|\varphi
       \right\|_{\alpha,p}.
    \end{align*}
    On the other hand, according to (\ref{L1-Linfty}) and H\"oolder's inequality, we obtain
    \begin{align*}
       \left\|e^{-z|y|^2}\mathcal{F}_{W,\alpha}(\varphi\chi_{B_\rho})\right\|_{\alpha,q}\leq  & \left\|e^{-z|y|^2}\right\|_{\alpha,q}\left\|\mathcal{F}_{W,\alpha}(\varphi\chi_{B_\rho})
       \right\|_{\alpha,\infty} \\
       \leq &  \left\|e^{-z|y|^2}\right\|_{\alpha,q}\left\|\varphi\chi_{B_\rho)}\right\|_{\alpha,1}\\
        \leq &  \left\|e^{-z|y|^2}\right\|_{\alpha,q}\left\||x|^{-s}\chi_{B_\rho)}\right\|_{\alpha,q} \left\||x|^s\varphi\right\|_{\alpha,p}.
    \end{align*}
  According to integral relationship (\ref{radialweinstein}), we have the following identity
$$ \left\|e^{-z|y|^2}\right\|_{\alpha,q}=\frac{1}{(2q)^{(\alpha+\frac{d}{2}+1)}}z^{-(\alpha+\frac{d}{2}+1)}
\quad\text{and}\quad \left\||x|^{-s}\chi_{B_\rho}\right\|_{\alpha,q}=K_\alpha\rho^{-s+(2\alpha+d+2)/q}.$$
Hence, we get
 $$\left\|e^{-z|y|^2}\mathcal{F}_{W,\alpha}(\varphi\chi_{B_\rho})\right\|_{\alpha,q}\leq \frac{C_\alpha}{(2q)^{(\alpha+\frac{d}{2}+1)}}z^{-(\alpha+\frac{d}{2}+1)}\rho^{-s+(2\alpha+d+2)/q}  \left\||x|^s\varphi\right\|_{\alpha,p},$$
 and
\begin{align*}
  \left\|e^{-z|y|^2}\mathcal{F}_{W,\alpha}(\varphi)\right\|_{\alpha,q}  \leq & \left\|e^{-z|y|^2}\mathcal{F}_{W,\alpha}(\varphi\chi_{B_\rho})\right\|_{\alpha,q} +  \left\|e^{-z|y|^2}\mathcal{F}_{W,\alpha}(\varphi\chi_{B_\rho^c})\right\|_{\alpha,q} \\
    \leq & \rho^{-s}\left(1+\frac{K_\alpha}{(2q)^{(\alpha+\frac{d}{2}+1)\frac{1}{q}}}\rho^{(2\alpha+d+2)/q}
    z^{-(\alpha+\frac{d}{2}+1)\frac{1}{q}}\right) \left\||x|^s\varphi\right\|_{\alpha,p}.
\end{align*}
By choosing $\rho=z^{1/2}$ we get the result.
    \end{proof}
\begin{thm}
 Let $1<p\leq 2$, $q=p/(p-1)$,  $0<s<(2\alpha+d+2)/q$ and $t>0$, then for all $\varphi\in L^p_\alpha(\mathbb{R}^{d+1}_+)$,
 \begin{equation}\label{LpHPWU}
   \left\|\mathcal{F}_{W,\alpha}(\varphi)\right\|_{\alpha,q}  \leq K(s,t) \left\||x|^s\varphi\right\|_{\alpha,p}^\frac{t}{s+t}\left\||y|^t\mathcal{F}_{W,\alpha}(\varphi)\right\|_
   {\alpha,q}^\frac{s}{s+t},
 \end{equation}
 where $K(s,t)$ is a positive constant.
 \end{thm}
\begin{proof}
  Let $\varphi\in L^p_\alpha(\mathbb{R}^{d+1}_+)$ and $1<p\leq 2$, such that
  $$\left\||x|^s\varphi\right\|_{\alpha,p}+\left\||y|^t\mathcal{F}_{W,\alpha}(\varphi)\right\|_
   {\alpha,q}<\infty.$$
   Suppose that  $0<s<(2\alpha+d+2)/q$ and $t<2$. Then, by Lemma \ref{lemexpo} we have, for all $z>0$,
   \begin{align*}
     \left\|\mathcal{F}_{W,\alpha}(\varphi)\right\|_{\alpha,q}  \leq & \left\|e^{-z|y|^2}\mathcal{F}_{W,\alpha}(\varphi)\right\|_{\alpha,q}   + \left\|(1-e^{-z|y|^2})\mathcal{F}_{W,\alpha}(\varphi)\right\|_{\alpha,q} \\
     \leq & \left(1+\frac{K_\alpha}{(2q)^{(\alpha+\frac{1}{d}+1)}\frac{1}{q}}\right)z^{-s/2} \left\||x|^s\varphi\right\|_{\alpha, p} + \left\|(1-e^{-z|y|^2})\mathcal{F}_{W,\alpha}(\varphi)\right\|_{\alpha,q}.
   \end{align*}
   on the other hand, we have
   $$\left\|(1-e^{-z|y|^2})\mathcal{F}_{W,\alpha}(\varphi)\right\|_{\alpha,q}=
   z^{t/2}\left\|(s|y|^2)^{-s/2}(1-e^{-z|y|^2})|y|^2\mathcal{F}_{W,\alpha}(\varphi)\right\|_{\alpha,q}.$$
   Take into account$(1-e^{-u})u^{-s/2}$ is bounded for all $u\geq 0$, if $t\leq 2$. Consequently,
   $$ \left\|\mathcal{F}_{W,\alpha}(\varphi)\right\|_{\alpha,q}  \leq K\left(z^{s/2}\left\||x|^s\varphi\right\|_{\alpha, p} +z^{t/2}\left\||y|^s\mathcal{F}_{W,\alpha}(\varphi)\right\|_{\alpha,q}\right).$$
   By taking $$z=\left(\frac{s\left\||x|^s\varphi\right\|_{\alpha, p}}{t\left\||y|^s\mathcal{F}_{W,\alpha}(\varphi)\right\|_{\alpha,q}}\right)^{\frac{2}{s+t}},$$
   we get the result for all $t\leq 2.$

   It remains to show the result for $t>0$. Since, we have for all $\varepsilon>0$,
   $|y|\leq \varepsilon+\varepsilon^{1-t}|y|^t,$
   then it follows that
   \begin{equation}\label{ineq1}
   \left\||y|\mathcal{F}_{W,\alpha}(\varphi)\right\|_{\alpha,q}\leq \varepsilon\left\|\mathcal{F}_{W,\alpha}(\varphi)\right\|_{\alpha,q} + \varepsilon^{1-t} \left\||y|^t\mathcal{F}_{W,\alpha}(\varphi)\right\|_{\alpha,q}.
   \end{equation}
By choosing,
          $$ \varepsilon=(t-1)^\frac{1}{t}\left(\frac{ \left\||y|^t\mathcal{F}_{W,\alpha}(\varphi)\right\|_{\alpha,q}}{\left\|\mathcal{F}_{W,\alpha}(\varphi)
          \right\|_{\alpha,q}}\right)^\frac{1}{t},$$
we get
 \begin{equation}\label{inq2}
 \left\||y|\mathcal{F}_{W,\alpha}(\varphi)\right\|_{\alpha,q}\leq \frac{t}{t-1}(t-1)^\frac{1}{t} \left\|\mathcal{F}_{W,\alpha}(\varphi)\right\|_{\alpha,q}^\frac{t-1}{t} \left\||y|^t\mathcal{F}_{W,\alpha}(\varphi)\right\|_{\alpha,q}^\frac{1}{t}.
 \end{equation}
 Combining the inequalities (\ref{ineq1}) and (\ref{inq2}), we obtain
 \begin{align*}
    \left\|\mathcal{F}_{W,\alpha}(\varphi)\right\|_{\alpha,q} & \leq  K\left\||x|^s\varphi\right\|_{\alpha,p}^\frac{1}{1+s}\left\||y|\mathcal{F}_{W,\alpha}(\varphi)\right\|_
   {\alpha,q}^\frac{s}{1+s} \\
    & \leq   K  \left\|\mathcal{F}_{W,\alpha}(\varphi)\right\|_{\alpha,q}^\frac{s(t-1)}{t(s+1)} \left\||x|^s\varphi\right\|_{\alpha,p}^\frac{1}{1+s}\left\||y|^t\mathcal{F}_{W,\alpha}(\varphi)\right\|_
   {\alpha,q}^\frac{s}{t(s+1)}.
 \end{align*}

 which implies
\begin{equation*}
   \left\|\mathcal{F}_{W,\alpha}(\varphi)\right\|_{\alpha,q}^\frac{s+t}{t(s+1)}
    \leq K  \left\||x|^s\varphi)\right\|_{\alpha,p}^\frac{1}{1+s} \left\||y|^s\mathcal{F}_{W,\alpha}(\varphi)\right\|_
   {\alpha,q}^\frac{s}{t(s+1)}.
\end{equation*}
 which gives the rsult for $t>0.$
\end{proof}
\begin{rem}
 Let $q=2$. According to Placherel formula (\ref{Plancherel formula}), we get
 \begin{equation*}
   \left\|\mathcal{F}_{W,\alpha}(\varphi)\right\|_{\alpha,2}  \leq K(s,t) \left\||x|^s\varphi\right\|_{\alpha,2}^\frac{t}{s+t}\left\||y|^t\mathcal{F}_{W,\alpha}(\varphi)\right\|_
   {\alpha,2}^\frac{s}{s+t},
 \end{equation*}
 which is the general case of the Heisenberg-Pauli-Weyl inequality (\ref{firstuncert}) proved by Mejjaoli \cite{mejjaoli2011uncertainty}.
\end{rem}

\section{$L^p$-Donoho-Strak uncertainty principles}
\begin{defn}
  Let $\Omega$ and $\Sigma$ be a measurable subsets of $\mathbb{R}^{d+1}_+$. We define the timelimiting
operator $P_\Omega$ by $P_\Omega\varphi:=\varphi\chi_\Omega$ and define the Weinstein integral operator $Q_\Sigma$ by $\mathcal{F}_{W,\alpha}(Q_\Sigma\varphi)=\mathcal{F}_{W,\alpha}(\varphi)\chi_\Sigma.$
\end{defn}
\begin{prop}\label{repQ}
  Let $\varphi\in L^p_\alpha(\mathbb{R}^{d+1}_+)$, $1\leq p\leq 2.$ If $\mu_\alpha(\Sigma)<\infty$, then we have
  \begin{equation*}
    Q_\Sigma\varphi(x)=\int_{\Sigma}\Lambda_{\alpha}^d(x, \lambda)\mathcal{F}_{W,\alpha}(\varphi)(\lambda)d\mu_{\alpha}(\lambda).
  \end{equation*}
\end{prop}
\begin{proof}
   Let $\varphi\in L^p_\alpha(\mathbb{R}^{d+1}_+)$, $1\leq p\leq 2$ and put $q=p/(p-1)$. According to inequalities (\ref{normLambda}) and (\ref{Lp-Lq}) we obtain
   \begin{align*}
     \left\|\mathcal{F}_{W,\alpha}(\varphi)\chi_\Sigma\right\|_{\alpha,1} & = \int_{\Sigma}|\mathcal{F}_{W,\alpha}(\varphi)(x)|d\mu_{\alpha}(x) \\
      & \leq \left(\mu_{\alpha}(\Sigma)\right)^{1/p} \left\|\mathcal{F}_{W,\alpha}(\varphi)\right\|_{\alpha,q}\\
       & \leq \left(\mu_{\alpha}(\Sigma)\right)^{1/p} \left\|\varphi\right\|_{\alpha,q}.
   \end{align*}
   In the same way, we get
   \begin{align*}
     \left\|\mathcal{F}_{W,\alpha}(\varphi)\chi_\Sigma\right\|_{\alpha,2} & =
     \left(\int_{\Sigma}|\mathcal{F}_{W,\alpha}(\varphi)(x)|^2d\mu_{\alpha}(x)\right)^\frac{1}{2} \\
      & \leq \left(\mu_{\alpha}(\Sigma)\right)^\frac{q-2}{2q} \left\|\mathcal{F}_{W,\alpha}(\varphi)\right\|_{\alpha,q}\\
       & \leq \left(\mu_{\alpha}(\Sigma)\right)^\frac{q-2}{2q} \left\|\varphi\right\|_{\alpha,q}.
   \end{align*}
Hence, $\mathcal{F}_{W,\alpha}(\varphi)\chi_\Sigma\in L^1_\alpha\cap L^2_\alpha(\mathbb{R}^{d+1}_+) $ and by the  definition of  the Weinstein integral operator $Q_\Sigma$  we obtain
$$Q_\Sigma\varphi=\mathcal{F}_{W,\alpha}^{-1}(\mathcal{F}_{W,\alpha}(\varphi)\chi_\Sigma).$$
Finally,  the inversion formula (\ref{inv}) gives the result.
\end{proof}
It is easy to see the following result by inequality (\ref{Lp-Lq}).
\begin{lem}
    Let $\varphi\in L^p_\alpha(\mathbb{R}^{d+1}_+)$, $1< p\leq 2$ and let $q=p/(p-1)$. Then we have
     $$\left\|\mathcal{F}_{W,\alpha}(Q_\Sigma\varphi)\right\|_{\alpha,q} \leq
      \left\|\varphi\right\|_{\alpha,p}.$$
\end{lem}
\begin{thm}\label{FPQ}
   Let $\Omega$ and $\Sigma$ be a measurable subsets of $\mathbb{R}^{d+1}_+$.   Let $\varphi\in L^p_\alpha(\mathbb{R}^{d+1}_+)$, $1< p\leq 2$ and $q=p/(p-1)$. Then we have
   \begin{equation*}
     \left\|\mathcal{F}_{W,\alpha}(Q_\Sigma P_\Omega\varphi\right\|_{\alpha,q} \leq \left(\mu_{\alpha}(\Sigma)\right)^\frac{1}{q} \left(\mu_{\alpha}(\Omega)\right)^\frac{1}{q}
      \left\|\varphi\right\|_{\alpha,p}.
   \end{equation*}
\end{thm}
\begin{proof}
  Assume that $ \mu_{\alpha}(\Sigma)<\infty$ and  $ \mu_{\alpha}(\Omega)<\infty$.  Let $\varphi\in L^p_\alpha(\mathbb{R}^{d+1}_+)$, $1\leq p\leq 2$ and $q=p/(p-1)$. From the definition of the Weinstein integral operator $Q_\Sigma$, we get
  $$\mathcal{F}_{W,\alpha}(Q_\Sigma P_\Omega\varphi)=\mathcal{F}_{W,\alpha}(P_\Omega\varphi)\chi_\Sigma.$$
  Therefore
  \begin{equation}\label{NPQ}
    \left\|\mathcal{F}_{W,\alpha}(Q_\Sigma P_\Omega\varphi)\right\|_{\alpha,q}  = \left(\int_{\Sigma}\left|\mathcal{F}_{W,\alpha}(P_\Omega\varphi)(\lambda)\right|^q d\mu_{\alpha}(\lambda)\right)^\frac{1}{q}.
  \end{equation}
  Since
  \begin{equation*}
    \mathcal{F}_{W,\alpha}(P_\Omega\varphi)(\lambda)=\int_{\Omega}\varphi(x)\Lambda_{\alpha}^d(x, \lambda
)d\mu_{\alpha}(x),
  \end{equation*}
 then according to Holder's inequality and (\ref{normLambda}), we obtain
  \begin{align*}
     \left|\mathcal{F}_{W,\alpha}(P_\Omega\varphi)\right| & \leq \left(\int_{\Omega}\left|\Lambda_{\alpha}^d(x, \lambda)\right|^q d\mu_{\alpha}(x)\right)^\frac{1}{q}  \left(\int_{\Omega}\left|\varphi(x)\right|^p d\mu_{\alpha}(x)\right)^\frac{1}{p}
 \\
     & \leq \left(\mu_{\alpha}(\Omega)\right)^\frac{1}{q}  \left\|\varphi\right\|_{\alpha,q}.
  \end{align*}
Hence, by (\ref{NPQ}), we get
 $$\left\|\mathcal{F}_{W,\alpha}(Q_\Sigma\varphi)\right\|_{\alpha,q} \leq
       \left(\mu_{\alpha}(\Sigma)\right)^\frac{1}{q} \left(\mu_{\alpha}(\Omega)\right)^\frac{1}{q}
      \left\|\varphi\right\|_{\alpha,p}.$$
      \end{proof}

\subsection{Concentration uncertainty principle}
In this section we present two continuous-time uncertainty
principles of concentration type and we show there are depend on the sets
of concentration $\Omega$ and $\Sigma$, and on the time function $\varphi$.
\begin{defn}
  Let $\Omega$, $\Sigma$ be a measurable subsets of $\mathbb{R}^{d+1}_+$ and  $\varphi\in L^p_\alpha(\mathbb{R}^{d+1}_+)$, $1\leq p\leq 2$.
  \begin{enumerate}
    \item[(i)] We say that $\varphi$ is $\varepsilon_\Omega$-concentrated to $\Omega$ in $ L^p_\alpha(\mathbb{R}^{d+1}_+)$-norm, if
         \begin{equation}\label{omegaconc}
         \left\|\varphi-P_\Omega\varphi\right\|_{\alpha,p}\leq  \varepsilon_\Omega\left\|\varphi\right\|_{\alpha,p}.
         \end{equation}
    \item[(ii)] $\mathcal{F}_{W,\alpha}(\varphi)$ is $\varepsilon_\Sigma$-concentrated to $\Sigma$ in $ L^q_\alpha(\mathbb{R}^{d+1}_+)$-norm, $q=p/(p-1)$,
         if
       \begin{equation}\label{Sigmaconc}
       \left\|\mathcal{F}_{W,\alpha}(\varphi)-\mathcal{F}_{W,\alpha}(Q_\Sigma\varphi)\right\|_{\alpha,q}\leq  \varepsilon_\Sigma\left\|\mathcal{F}_{W,\alpha}(\varphi)\right\|_{\alpha,q}.
    \end{equation}
  \end{enumerate}
\end{defn}
The following theorem, states the first continuous-time uncertainty
principle of concentration type for the $L^p_\alpha$-theory.
\begin{thm}\label{unc1}
 Let $\Omega$, $\Sigma$ be a measurable subsets of $\mathbb{R}^{d+1}_+$ and  $\varphi\in L^p_\alpha(\mathbb{R}^{d+1}_+)$, $1< p\leq 2$. If $\varphi$ is $\varepsilon_\Omega$-concentrated to $\Omega$ in $ L^p_\alpha(\mathbb{R}^{d+1}_+)$-norm and  $\mathcal{F}_{W,\alpha}(\varphi)$ is $\varepsilon_\Sigma$-concentrated to $\Sigma$ in $ L^q_\alpha(\mathbb{R}^{d+1}_+)$-norm, $q=p/(p-1)$, then we have
 \begin{equation*}
    \left\|\mathcal{F}_{W,\alpha}(\varphi)\right\|_{\alpha,q}\leq \frac{ \left(\mu_{\alpha}(\Sigma)\right)^\frac{1}{q} \left(\mu_{\alpha}(\Omega)\right)^\frac{1}{q}+\varepsilon_\Omega}{1-\varepsilon_\Sigma}\left\|\varphi
    \right\|_{\alpha,p}.
 \end{equation*}
\end{thm}
\begin{proof}
  Let $\varphi\in L^p_\alpha(\mathbb{R}^{d+1}_+)$, $1< p\leq 2$. According to (\ref{omegaconc}), (\ref{Sigmaconc}) and Theorem \ref{FPQ} it follows that
  \begin{align*}
     \left\|\mathcal{F}_{W,\alpha}(\varphi)-\mathcal{F}_{W,\alpha}(Q_\Sigma P_\Omega\varphi)\right\|_{\alpha,q}  \leq &  \left\|\mathcal{F}_{W,\alpha}(\varphi)-\mathcal{F}_{W,\alpha}(Q_\Sigma \varphi)\right\|_{\alpha,q}\\ & +  \left\|\mathcal{F}_{W,\alpha}(Q_\Sigma \varphi)-\mathcal{F}_{W,\alpha}(Q_\Sigma P_\Omega\varphi)\right\|_{\alpha,q}\\
      \leq & \varepsilon_\Sigma \left\|\mathcal{F}_{W,\alpha}(\varphi)\right\|_{\alpha,q}+  \left\|\varphi-P_\Omega\varphi\right\|_{\alpha,p}\\
      \leq &\varepsilon_\Sigma \left\|\mathcal{F}_{W,\alpha}(\varphi)\right\|_{\alpha,q}+ \varepsilon_\Omega\left\|\varphi\right\|_{\alpha,p}.
  \end{align*}
  Applying the triangle inequality and think to Theorem \ref{FPQ}, we show that
  \begin{align*}
    \left\|\mathcal{F}_{W,\alpha}(\varphi)\right\|_{\alpha,q} & \leq \left\|\mathcal{F}_{W,\alpha}(Q_\Sigma P_\Omega\varphi)\right\|_{\alpha,q}+\left\|\mathcal{F}_{W,\alpha}(\varphi)-\mathcal{F}_{W,\alpha}(Q_\Sigma P_\Omega\varphi)\right\|_{\alpha,q}\\
     & \leq \left( \left(\mu_{\alpha}(\Sigma)\right)^\frac{1}{q} \left(\mu_{\alpha}(\Omega)\right)^\frac{1}{q}+\varepsilon_\Omega\right)\left\|\varphi
    \right\|_{\alpha,p}+\varepsilon_\Sigma\left\|\mathcal{F}_{W,\alpha}(\varphi)\right\|_{\alpha,q}.
  \end{align*}
  which gives the desired result.
\end{proof}
The following result gives the second continuous-time uncertainty principle of
concentration type for the $L^1_\alpha\cap L^p_\alpha$ theory.
\begin{thm}\label{unc2}
   Let $\Omega$, $\Sigma$ be a measurable subsets of $\mathbb{R}^{d+1}_+$ and  $\varphi\in (L^1_\alpha\cap L^p_\alpha)(\mathbb{R}^{d+1}_+)$, $1< p\leq 2$. If $\varphi$ is $\varepsilon_\Omega$-concentrated to $\Omega$ in $ L^p_\alpha(\mathbb{R}^{d+1}_+)$-norm and  $\mathcal{F}_{W,\alpha}(\varphi)$ is $\varepsilon_\Sigma$-concentrated to $\Sigma$ in $ L^q_\alpha(\mathbb{R}^{d+1}_+)$-norm, $q=p/(p-1)$, then we have
   \begin{equation*}
    \left\|\mathcal{F}_{W,\alpha}(\varphi)\right\|_{\alpha,q}\leq \frac{ \left(\mu_{\alpha}(\Sigma)\right)^\frac{1}{q} \left(\mu_{\alpha}(\Omega)\right)^\frac{1}{q}}{(1-\varepsilon_\Omega)(1-\varepsilon_\Sigma)}\left\|\varphi
    \right\|_{\alpha,p}.
 \end{equation*}
\end{thm}
\begin{proof}
   Let $\Omega$, $\Sigma$ be a measurable subsets of $\mathbb{R}^{d+1}_+$ such that $\mu_{\alpha}(\Omega)<\infty$ and  $\mu_{\alpha}(\Sigma)<\infty$. Assume that $\varphi\in (L^1_\alpha\cap L^p_\alpha)(\mathbb{R}^{d+1}_+)$, $1< p\leq 2$. Like that  $\mathcal{F}_{W,\alpha}(\varphi)$ is $\varepsilon_\Sigma$-concentrated to $\Sigma$ in $ L^q_\alpha(\mathbb{R}^{d+1}_+)$-norm, $q=p/(p-1)$,then we get
   \begin{align*}
     \left\|\mathcal{F}_{W,\alpha}(\varphi)\right\|_{\alpha,q} &\leq  \varepsilon_\Sigma\left\|\mathcal{F}_{W,\alpha}(\varphi)\right\|_{\alpha,q}+ \left\|\mathcal{F}_{W,\alpha}(Q_\Sigma\varphi)\right\|_{\alpha,q}\\
      & \leq \varepsilon_\Sigma\left\|\mathcal{F}_{W,\alpha}(\varphi)\right\|_{\alpha,q}+ \left(\mu_{\alpha}(\Sigma)\right)^\frac{1}{q}\left\|\mathcal{F}_{W,\alpha}(\varphi)\right\|_{\alpha,\infty}.
   \end{align*}
According to inequality (\ref{L1-Linfty}), it follows
\begin{equation}\label{conc1}
    \left\|\mathcal{F}_{W,\alpha}(\varphi)\right\|_{\alpha,q}\leq \frac{ \left(\mu_{\alpha}(\Sigma)\right)^\frac{1}{q}} {1-\varepsilon_\Sigma}\left\|\varphi
    \right\|_{\alpha,1}.
\end{equation}
On the other hand, since $\varphi$ is $\varepsilon_\Omega$-concentrated to $\Omega$ in $ L^p_\alpha(\mathbb{R}^{d+1}_+)$-norm, then
 \begin{align*}
     \left\|\varphi\right\|_{\alpha,1} &\leq  \varepsilon_\Omega\left\|\varphi\right\|_{\alpha,1}+ \left\|P_\Sigma\varphi\right\|_{\alpha,1}\\
      & \leq \varepsilon_\Omega\left\|\varphi\right\|_{\alpha,q}+ \left(\mu_{\alpha}(\Omega)\right)^\frac{1}{p}\left\|\varphi\right\|_{\alpha,p}.
   \end{align*}
Hence
\begin{equation}\label{conc2}
 \left\|\varphi\right\|_{\alpha,1} \leq \frac{\left(\mu_{\alpha}(\Omega)\right)^\frac{1}{p}}{1-\varepsilon_\Omega}\left\|\varphi\right\|_{\alpha,q}.
\end{equation}
Combining (\ref{conc1}) and (\ref{conc2}) we obtain the result of this theorem.
\end{proof}
\begin{rem}
  We observe at first that the statement of Theorem \ref{unc1} depends on the time function $\varphi$. However although for $p=q=2$, the continuous-time uncertainty principle becomes
  \begin{equation*}
    1-\varepsilon_\Omega-\varepsilon_\Sigma\leq \mu_\alpha(\Omega)^\frac{1}{2}\mu_\alpha(\Sigma)^\frac{1}{2},
  \end{equation*}
  and is independent on the time function $\varphi$. Likewise we observe  that the statement of Theorem \ref{unc2} depends on the time function $\varphi$. Like the first  continuous-time uncertainty principle
  the statement of Theorem \ref{unc2} is independent on the time function $\varphi$ for $p=q=2$, and we have
   \begin{equation*}
    (1-\varepsilon_\Omega)(1-\varepsilon_\Sigma)\leq \mu_\alpha(\Omega)^\frac{1}{2}\mu_\alpha(\Sigma)^\frac{1}{2}.
  \end{equation*}
\end{rem}
\subsection{Continuous-bandlimited uncertainty principle}
In this section, we establish continuous-bandlimited uncertainty principle of
concentration. This principle depends  on the sets of concentration
$\Omega$ and $\Sigma$, but he is independent on the bandlimited function $\varphi$.
\begin{defn}
  Let $1\leq p\leq 2$ and $\psi\in L^p(\mathbb{R}_+^d)$.
  \begin{enumerate}
    \item[(i)] We say that $\psi$ is bandlimited to $\Sigma$ if
  $Q_\Sigma\psi=\psi$ and denote by $\mathcal{B}_\alpha^p(\Sigma)$ the set of functions $\psi\in L^p(\mathbb{R}_+^d)$ that are bandlimited to $\Sigma$.
    \item [(ii)] We say that $\varphi$ is $\varepsilon_\Sigma$-bandlimited to $\Sigma$ in $L^p_\alpha$-norm  if there exists $\psi\in \mathcal{B}_\alpha^p(\Sigma)$ such that
        \begin{equation*}
          \left\|\varphi-\psi\right\|_{\alpha,p}\leq \varepsilon  \left\|\varphi\right\|_{\alpha,p}.
        \end{equation*}
  \end{enumerate}
\end{defn}

 The space of bandlimited functions $\mathcal{B}_\alpha^p(\Sigma)$ satisfies the following property.
\begin{prop}\label{propband}
   Let $\Omega$, $\Sigma$ be a measurable subsets of $\mathbb{R}^{d+1}_+$. For all $\psi\in\mathcal{B}_\alpha^p(\Sigma)$,  $1\leq p\leq 2$, we have
   \begin{equation*}
     \left\|P_\Omega\psi\right\|_{\alpha,p}\leq  \left(\mu_{\alpha}(\Sigma)\right)^\frac{1}{p} \left(\mu_{\alpha}(\Omega)\right)^\frac{1}{p}  \left\|\psi\right\|_{\alpha,p}.
   \end{equation*}
\end{prop}
\begin{proof}
 The estimation is  trivial if the measure of $\Omega$ or $\Sigma$ is infinite.  Suppose that $ \mu_{\alpha}(\Sigma)<\infty$ and  $ \mu_{\alpha}(\Omega)<\infty$. Let $\psi\in \mathcal{B}_\alpha^p(\Sigma)$ such that  $1\leq p\leq 2$, then from Propostion \ref{repQ}, we have
 \begin{equation*}
    \psi(x)=\int_{\Sigma}\Lambda_{\alpha}^d(x, \lambda)\mathcal{F}_{W,\alpha}(\psi)(\lambda)d\mu_{\alpha}(\lambda).
  \end{equation*}
  From the property (\ref{normLambda}) and H\"older's inequality, we obtain for all $q=p/(p-1)$
 \begin{align*}
   \psi(x) & \leq  \left(\mu_{\alpha}(\Sigma)\right)^\frac{1}{p} \left\|\mathcal{F}_{W,\alpha}(\psi)\right\|_{\alpha,q}\\
     & \leq \left(\mu_{\alpha}(\Sigma)\right)^\frac{1}{p} \left\|\psi\right\|_{\alpha,p}.
 \end{align*}
    Applying now  the time-limiting operator $P_\Omega$ to the bandlimited function $\psi$ and observing the $L^p_\alpha$-norm, we obtain
    \begin{align*}
  \left\|P_\Omega\psi\right\|_{\alpha,p} & =\left( \int_{\Omega}|\psi(x)|^pd\mu_{\alpha}(x)\right)^\frac{1}{p}\\
     & \leq \left(\mu_{\alpha}(\Omega)\right)^\frac{1}{p}\left(\mu_{\alpha}(\Sigma)\right)^\frac{1}{p} \left\|\psi\right\|_{\alpha,p}.
 \end{align*}
 which yields the desired result.
\end{proof}
\begin{thm}\label{secc}
 Let $\Omega$, $\Sigma$ be a measurable subsets of $\mathbb{R}^{d+1}_+$ and let $\varphi\in L^p(\mathbb{R}_+^d)$ such that  $1\leq p\leq 2$. Then if $\varphi$ is $\varepsilon_\Sigma$-bandlimited to $\Sigma$ in $L^p_\alpha$-norm, we have
  \begin{equation*}
     \left\|P_\Omega\psi\right\|_{\alpha,p}\leq  \left((1+\varepsilon_\Sigma)\left(\mu_{\alpha}(\Sigma)\right)^\frac{1}{p} \left(\mu_{\alpha}(\Omega)\right)^\frac{1}{p} +\varepsilon_\Sigma \right) \left\|\psi\right\|_{\alpha,p}.
   \end{equation*}
\end{thm}
\begin{proof}
  Let $\varphi\in L^p(\mathbb{R}_+^d)$ such that  $1\leq p\leq 2$. Like that $\varphi$ is $\varepsilon_\Sigma$-bandlimited to $\Sigma$ in $L^p_\alpha$-norm, then there exists $\psi\in \mathcal{B}_\alpha^p(\Sigma)$ such that
        \begin{equation*}
          \left\|\varphi-\psi\right\|_{\alpha,p}\leq \varepsilon  \left\|\varphi\right\|_{\alpha,p}
        \end{equation*}
        and we have
        \begin{equation}\label{Pfi}
         \left\|P_\Omega\varphi\right\|_{\alpha,p}  \leq   \left\|P_\Omega\psi\right\|_{\alpha,p}+ \left\|P_\Omega(\varphi-\psi)\right\|_{\alpha,p}
            \leq  \left\|P_\Omega\psi\right\|_{\alpha,p}+\varepsilon_\Sigma\left\|\varphi\right\|_{\alpha,p}.
        \end{equation}
        Observe now the result of Proposition \ref{propband} and the fact that

        \begin{equation*}
          \left\|\psi\right\|_{\alpha,p}\leq (1+\varepsilon_\Sigma)\left\|\varphi\right\|_{\alpha,p}
        \end{equation*}
        we obtain the desired result.
\end{proof}
\begin{thm}
   Let $\Omega$, $\Sigma$ be a measurable subsets of $\mathbb{R}^{d+1}_+$ and let $\varphi\in L^p(\mathbb{R}_+^d)$ such that  $1\leq p\leq 2$. Then if $\varphi$ is  $\varepsilon_\Omega$-concentrated to $\Omega$ and $\varepsilon_\Sigma$-bandlimited to $\Sigma$ in $L^p_\alpha$-norm, we have
  \begin{equation*}
     1-\varepsilon_\Omega-\varepsilon_\Sigma\leq  (1+\varepsilon_\Sigma)\left(\mu_{\alpha}(\Sigma)\right)^\frac{1}{p} \left(\mu_{\alpha}(\Omega)\right)^\frac{1}{p}.
   \end{equation*}
\end{thm}
\begin{proof}
  Let $\varphi\in L^p(\mathbb{R}_+^d)$ such that  $1\leq p\leq 2$. Like that $\varphi$ is  $\varepsilon_\Omega$-concentrated to $\Omega$  in $L^p_\alpha$-norm, then by estimation (\ref{omegaconc}) we obtain
  \begin{equation*}
      \left\|\varphi\right\|_{\alpha,p}\leq   \varepsilon_\Omega\left\|\varphi\right\|_{\alpha,p}  +\left\|P_\Omega\varphi\right\|_{\alpha,p}.
  \end{equation*}
Therefore
\begin{equation*}
      \left\|\varphi\right\|_{\alpha,p}\leq  \frac{1}{1-\varepsilon_\Omega}\left\|P_\Omega\varphi\right\|_{\alpha,p}.
  \end{equation*}
  Finally, we deduce the desired estimation by inequality (\ref{Pfi}) and Theorem \ref{secc}.
\end{proof}
\begin{rem}
  The continuous bandlimited uncertainty principle of concentration type for the $L^p_\alpha$-norm given by pervious Corollary is independent on the bandlimited function $\varphi$ for every $1\leq p\leq 2$.
\end{rem}


\begin{thebibliography}{10}

\bibitem{nahia1996spherical}
Z.~{Ben Nahia} and N.~{Ben Salem}.
\newblock {Spherical harmonics and applications associated with the {{
  W}}einstein operator.}
\newblock In {\em {Potential theory -- ICPT '94. Proceedings of the
  international conference, Kouty, Czech Republic}}.

\bibitem{brelot1978equation}
M.~Brelot.
\newblock Equation de weinstein et potentiels de{{ M}}arcel {{ R}}iesz.
\newblock In {\em S{\'e}minaire de Th{\'e}orie du Potentiel Paris, No. 3},
  pages 18--38. Springer, 1978.

\bibitem{ciatti2007heisenberg}
P.~Ciatti, F.~Ricci, and M.~Sundari.
\newblock Heisenberg-{{P}}auli-{{W}}eyl uncertainty inequalities and polynomial
  volume growth.
\newblock {\em Advances in Mathematics}, 215(2):616--625, 2007.

\bibitem{donoho1989uncertainty}
D.~L. Donoho and P.~B. Stark.
\newblock Uncertainty principles and signal recovery.
\newblock {\em SIAM Journal on Applied Mathematics}, 49(3):906--931, 1989.

\bibitem{ghobber2013uncertainty}
S.~Ghobber.
\newblock Uncertainty principles involving {{ L}}$^1$-norms for the {{D}}unkl
  transform.
\newblock {\em Integral Transforms and Special Functions}, 24(6):491--501,
  2013.

\bibitem{ghobber2014variations}
S.~Ghobber.
\newblock Variations on uncertainty principles for integral operators.
\newblock {\em Applicable Analysis}, 93(5):1057--1072, 2014.

\bibitem{laeng1999uncertainty}
E.~Laeng and C.~Morpurgo.
\newblock An uncertainty inequality involving {{ L}}$^1$-norms.
\newblock {\em Proceedings of the American Mathematical Society},
  127(12):3565--3572, 1999.

\bibitem{landau1975szego}
H.~Landau.
\newblock On {{S}}zeg{\"o}'s eingenvalue distribution theorem and
  non-{{H}}ermitian kernels.
\newblock {\em Journal d\'Analyse Math{\'e}matique}, 28(1):335--357, 1975.

\bibitem{landau1962prolate}
H.~J. Landau and H.~O. Pollak.
\newblock Prolate spheroidal wave functions, {{F}}ourier analysis and
  uncertainty-iii: The dimension of the space of essentially time-and
  band-limited signals.
\newblock {\em Bell Labs Technical Journal}, 41(4):1295--1336, 1962.

\bibitem{mejjaoli2012weinstein}
H.~Mejjaoli and A.~O.~A. Salem.
\newblock Weinstein {{G}}abor transform and applications.
\newblock {\em Advances in Pure Mathematics}, 2(03):203, 2012.

\bibitem{mejjaoli2011uncertainty}
H.~Mejjaoli and M.~Salhi.
\newblock Uncertainty principles for the {{W}}einstein transform.
\newblock {\em Czechoslovak mathematical journal}, 61(4):941--974, 2011.

\bibitem{morpurgo2001extremals}
C.~Morpurgo.
\newblock Extremals of some uncertainty inequalities.
\newblock {\em Bulletin of the London Mathematical Society}, 33(1):52--58,
  2001.

\bibitem{nahia1996mean}
Z.~B. Nahia and N.~B. Salem.
\newblock On a mean value property associated with the {{W}}einstein operator.
\newblock In {\em Proceedings of the International Conference on Potential
  Theory held in Kouty, Czech Republic (ICPT'94)}, pages 243--253, 1996.

\bibitem{salem2015heisenberg}
N.~B. Salem and A.~R. Nasr.
\newblock Heisenberg-type inequalities for the {{W}}einstein operator.
\newblock {\em Integral Transforms and Special Functions}, 26(9):700--718,
  2015.

\bibitem{zbMATH06504466}
F.~{Soltani}.
\newblock {$L^{p}$ Donoho-Stark uncertainty principles for the Dunkl transform
  on $\mathbb{R}^{\text{d}}$.}
\newblock {\em {J. Phys. Math.}}, 5:4, 2014.

\bibitem{zbMATH06692262}
F.~{Soltani} and J.~{Ghazwani}.
\newblock {A variation of the $L^p$ uncertainty principles for the Fourier
  transform.}
\newblock {\em {Proc. Inst. Math. Mech., Natl. Acad. Sci. Azerb.}},
  42(1):10--24, 2016.

\bibitem{zbMATH03367521}
E.~M. {Stein} and G.~{Weiss}.
\newblock {Introduction to {{F}}ourier analysis on Euclidean spaces.}
\newblock {Princeton Mathematical Series. Princeton, N. J.: Princeton
  University Press. X, 297 p.}, 1971.

\bibitem{weinstein1962singular}
A.~Weinstein.
\newblock Singular partial differential equations and their applications.
\newblock {\em Fluid Dynamics and Applied Mathematics}, 67:29--49, 1962.

\end{thebibliography}
\end{document}